\documentclass[11pt,reqno]{amsproc}

\usepackage[T1]{fontenc}
\usepackage{lmodern}
\usepackage{mathrsfs}
\usepackage{microtype}
\usepackage[a4paper,margin=30mm]{geometry}
\usepackage{amsmath,amssymb,mathtools}
\usepackage{aliascnt}
\usepackage{hyperref}
\usepackage[nameinlink,noabbrev]{cleveref}

\hypersetup{
  hidelinks,
  pdftitle={Reduced Embry characterizations of quasinormal operators},
  pdfauthor={Pawel Pietrzycki}
}
\numberwithin{equation}{section}

\newtheorem{theorem}{Theorem}[section]
\newaliascnt{proposition}{theorem}
\newtheorem{proposition}[proposition]{Proposition}
\aliascntresetthe{proposition}
\newaliascnt{lemma}{theorem}
\newtheorem{lemma}[lemma]{Lemma}
\aliascntresetthe{lemma}
\newaliascnt{corollary}{theorem}

\aliascntresetthe{corollary}
\crefname{theorem}{theorem}{theorems}
\Crefname{theorem}{Theorem}{Theorems}
\crefname{proposition}{proposition}{propositions}
\Crefname{proposition}{Proposition}{Propositions}
\crefname{lemma}{lemma}{lemmas}
\Crefname{lemma}{Lemma}{Lemmas}
\crefname{corollary}{corollary}{corollaries}
\Crefname{corollary}{Corollary}{Corollaries}

\newcommand{\BH}{\boldsymbol B(\mathcal H)}
\newcommand{\N}{\mathbb N}

\newcommand*{\jd}[1]{{\mathscr N}(#1)}

\newcommand*{\ran}{\mathscr{R}}

\def\sslim{\qopname\relax m{{\sf{s}\textrm-}lim}}
\title[Reduced Embry characterizations]
{Reduced Embry characterizations\\
of quasinormal operators}
\author{Pawe\l{} Pietrzycki}
\address{Wydzia\l{} Matematyki i Informatyki, Uniwersytet Jagiello\'nski,
  ul. \L{}ojasiewicza 6, 30-348 Krak\'ow, Poland}
\email{pawel.pietrzycki@im.uj.edu.pl}
\date{}

\begin{document}
\begin{abstract}
We prove that, for any two integers $2\le m<n$, a bounded operator
$A$ on a complex Hilbert space satisfying
\[
  A^{*k}A^k=(A^*A)^k,
  \qquad k\in\{m,n\},
\]
is quasinormal. More generally, we show that the inequalities
\[
  A^{*m}A^m\ge(A^*A)^m
  \quad \text{and} \quad
  A^{*n}A^n\le(A^*A)^n
\]
already imply quasinormality. Together with the known
single-exponent counterexamples, this completely resolves the
reduced Embry problem for bounded operators: the identities
$A^{*k}A^k=(A^*A)^k$, $k\in S\subseteq\{1,2,\ldots\}$,
characterize quasinormality if and only if $S$ contains at least
two distinct integers greater than one. We also resolve the flatness problem for operators of class $\mathcal A$:
equality between two consecutive normalized moments
$(A^{*n}A^n)^{1/n}$ and $(A^{*(n+1)}A^{n+1})^{1/(n+1)}$,
with $n\ge2$, implies quasinormality.
Both results rely on an interpolation estimate obtained either
from Hansen's inequality or from the transformer inequality
for weighted operator geometric means.
The solution to the reduced Embry problem additionally uses
an extremal argument for optimal power bounds.
\end{abstract}

\maketitle

\section{Introduction}\label{sec:introduction}

A bounded operator $A$ on a (complex) Hilbert space $\mathcal H$ is
\emph{quasinormal} if $A(A^*A)=(A^*A)A$. This class was introduced
by Brown \cite{Brown}. In her work on the Halmos--Bram criterion
for subnormality, Embry \cite[p.~63]{Embry} proved that $A$ is
quasinormal if and only if
\begin{equation}\label{wst}
  A^{*k}A^k=(A^*A)^k,
  \qquad k\in\N,
\end{equation}
where $\N$ denotes the set of positive integers. The reduced Embry problem asks which
subfamilies of \eqref{wst} already imply quasinormality (see, e.g., \cite[Problem~1.4]{PietrzyckiStochel2021JFA}).
We call a set $S\subseteq\N$ \emph{forcing} if every bounded
operator on a Hilbert space satisfying
$A^{*k}A^k=(A^*A)^k$ for all $k\in S$ is quasinormal. This problem has been studied by several authors
(see \cite{Jib,JablonskiJungStochel2014,pat,Pietrzycki2016,Pietrzycki2018,PietrzyckiStochel2023,uch93,uch2}).

As noted in \cite[p.~135]{JablonskiJungStochel2014}, the fact that
$\{2,3\}$ is forcing has been known to specialists since the late
1980s. In 1993 Uchiyama \cite[Theorem~2.1]{uch93} proved that, on
separable Hilbert spaces, the identities indexed by
$\{k,k+1,\ell,\ell+1\}$ imply quasinormality whenever
$k,\ell\in\N$ and $k<\ell$. The present author generalized this
result in \cite[Theorem~3.11]{Pietrzycki2018}, showing that
$\{p,m,n,m+p,n+p\}$ is forcing for all $p,m,n\in\N$ with $m<n$,
without any separability assumption. It was also known that the sets $\{2,n\}$, $n\ge3$, are forcing;
this follows from \cite[Theorem~4.1]{PietrzyckiStochel2023}, since
the identity at exponent $2$ places the operator in class $\mathcal{A}$. Recall that an operator $A\in\boldsymbol B(\mathcal H)$ is of
\emph{class $\mathcal{A}$} if $A^*A\le(A^{*2}A^2)^{1/2}$.
A useful feature of operators of class $\mathcal{A}$ is the monotonicity of
their normalized moments (see \cite[Theorem~1]{Ito02}): if $A$ is of class $\mathcal{A}$, then
\begin{equation}\label{eq}
A^*A
\le (A^{*2}A^2)^{1/2}
\le (A^{*3}A^3)^{1/3}
\le \cdots .
\end{equation}
This monotonicity naturally leads to the corresponding flatness
problem. Namely, one asks whether, for an operator $A$ of class
$\mathcal A$, the equality
\begin{equation}\label{eq:flatness-intro}
  (A^{*n}A^n)^{1/n}
  =
  (A^{*(n+1)}A^{n+1})^{1/(n+1)}
\end{equation}
between two consecutive normalized moments forces $A$ to be
quasinormal. This question was posed in
\cite[Problem~4.2]{PietrzyckiStochel2023}.

Under additional assumptions on the operator, a single identity
may suffice. An early result in this direction is due to Patel
\cite[Theorem~5]{pat}, who proved that if a partial isometry $A$
satisfies $A^{*k}A^k\ge (A^*A)^k$
for some integer $k\ge2$, then $A$ is quasinormal. Uchiyama \cite[p.~198]{uch93}
showed that a compact operator satisfying
$A^{*k}A^k=(A^*A)^k$ for some integer $k\ge2$ is normal.
For arbitrary bounded operators, however, no single nontrivial
exponent suffices. Jab\l{}o\'nski, Jung and Stochel
(see \cite[Example~5.5]{JablonskiJungStochel2014}) constructed, for each
integer $k\ge2$, a bounded non-quasinormal operator satisfying
$A^{*k}A^k=(A^*A)^k$ (see also \cite{Pietrzycki2016}).
Since the identity at exponent $1$ holds for every bounded operator,
any forcing set must therefore contain at least two distinct
integers greater than one.

Our main result gives a complete solution to the reduced Embry problem
for bounded operators: any two distinct exponents greater than one
force quasinormality.

\begin{theorem}\label{cor:two-exponents}
Let $2\le m<n$ be integers and let $A\in\BH$. Then $A$ is
quasinormal if and only if
\[
  A^{*k}A^k=(A^*A)^k,
  \qquad k\in\{m,n\}.
\]
\end{theorem}

Together with the single-exponent counterexamples discussed above,
Theorem~\ref{cor:two-exponents} yields a complete classification:
a set $S\subseteq\N$ is forcing if and only if
$\#(S\setminus\{1\})\ge2$.
In particular, no arithmetic relation between the two exponents
is required.

We prove the following stronger result, which shows that the two
identities can be replaced by inequalities in opposite directions.

\begin{theorem}\label{thm:all-pairs}
Let $2\le m<n$ be integers and let $A\in\BH$. Then $A$ is
quasinormal if and only if
\begin{equation}\label{eq:main-hypotheses}
  A^{*m}A^m\ge(A^*A)^m,
  \qquad
  A^{*n}A^n\le(A^*A)^n.
\end{equation}
\end{theorem}

Thus a lower bound at the smaller exponent and an upper bound at the
larger exponent force all the identities in Embry's characterization.

Our second main result resolves the flatness problem for operators
of class $\mathcal A$ posed in
\cite[Problem~4.2]{PietrzyckiStochel2023}.
It shows that equality between two consecutive normalized moments,
starting at any exponent greater than one, forces quasinormality.

\begin{theorem}\label{thm:flatness}
Let $A\in\BH$ be of class $\mathcal A$, and let $n\ge2$ be an integer.
If
\begin{equation}\label{eq:flatness-main}
  (A^{*n}A^n)^{1/n}
  =
  (A^{*(n+1)}A^{n+1})^{1/(n+1)},
\end{equation}
then $A$ is quasinormal.
\end{theorem}

The restriction $n\ge2$ is necessary.
Indeed, there exist non-quasinormal operators $A$ satisfying
$A^{*2}A^2=(A^*A)^2$; see
\cite[(2.3)]{PietrzyckiStochel2023}.
Every such operator belongs to class $\mathcal A$ and satisfies
\eqref{eq:flatness-main} with $n=1$.

Theorems~\ref{thm:all-pairs} and~\ref{thm:flatness} are proved
in Sections~\ref{sec:consecutive} and~\ref{sec:flatness},
respectively.

The proofs of Theorems~\ref{thm:all-pairs} and~\ref{thm:flatness}
rely on a common interpolation estimate for operator moments.
We give two proofs of this estimate: one uses the transformer
inequality and monotonicity of the weighted operator geometric mean,
while the other relies directly on Hansen's inequality.
For Theorem~\ref{thm:all-pairs}, the key step is an extremal argument
for the optimal exponents in power bounds for $A^{*r}A^r$.
The interpolation estimate imposes relations among these exponents
that reduce the general case to consecutive exponents.
For Theorem~\ref{thm:flatness}, the same estimate yields the two
moment identities needed to apply the quasinormality criterion
recalled in the preliminaries.
\subsection*{AI Disclosure} All ideas, proof strategies, and mathematical contributions in this manuscript originated exclusively with the author. OpenAI’s ChatGPT (GPT-5.6) was used solely to accelerate calculations. The author takes full responsibility for the content of the publication.
\subsection*{Acknowledgements}
The author thanks OpenAI for providing access to ChatGPT through
the ChatGPT for Academic Researchers program.

\section{Preliminaries}\label{sec:preliminaries}

Throughout the paper, all Hilbert spaces are complex and all operators
are bounded and linear. For a Hilbert space $\mathcal H$, we denote
by $\boldsymbol B(\mathcal H)$ the $C^*$-algebra of all operators on
$\mathcal H$ and by $I$ its identity. For $T\in\boldsymbol B(\mathcal H)$,
we write $|T|=(T^*T)^{1/2}$ and denote its adjoint, kernel, range and
spectrum 
by $T^*$, $\jd T$, $\ran (T)$ and $\sigma(T)$,
respectively. Powers of positive operators are understood in the sense of the
continuous functional calculus.
Recall that every $A\in\boldsymbol B(\mathcal H)$ admits a polar
decomposition $A=U|A|$, where $U$ is the unique partial
isometry satisfying $\ker U=\ker A$.

Let $J\subseteq\mathbb R$ be an interval. A continuous function
$f\colon J\to\mathbb R$ is said to be \emph{operator monotone}
if $f(T)\le f(S)$ whenever $T,S$ are selfadjoint operators on
an arbitrary Hilbert space such that $T\le S$ and their spectra
are contained in $J$ (see \cite{Sim19}).
By the L\"owner--Heinz inequality, the function $t\mapsto t^p$
is operator monotone on $[0,\infty)$ for every $0<p\le1$.
We shall also use the following inequality of Hansen.

\begin{theorem}[{\cite{Han80}}]\label{hans}
Let $A\in\boldsymbol B(\mathcal H)$ be positive,
$T\in\boldsymbol B(\mathcal H)$ be a contraction, and
$f\colon[0,\infty)\to\mathbb R$ be a continuous operator monotone
function such that $f(0)\ge0$. Then
\begin{equation}\label{Han-inq}
  T^*f(A)T\le f(T^*AT).
\end{equation}

\end{theorem}

The following criterion for quasinormality was established
by the present author in \cite{Pietrzycki2018} under
the additional assumption that $B$ is injective. The version
without this assumption is given in
\cite{PietrzyckiStochel2023} (see also
\cite[Remark~3.3]{PietrzyckiStochel2023}).

\begin{theorem}[{\cite[Lemma~3.7]{Pietrzycki2018},\cite[Theorem~1.6]{PietrzyckiStochel2023}}]\label{bblem}
Let $T,B\in\boldsymbol B(\mathcal H)$, with $B\ge0$, and let
$\alpha,\beta>0$ be distinct. If
\[
  T^*T\le B
  \quad\text{and}\quad
  T^*B^sT=B^{s+1},
  \qquad s\in\{\alpha,\beta\},
\]
then $B=T^*T$ and $T$ is quasinormal.
\end{theorem}

\section{Proof of Theorem~\ref{thm:all-pairs}}\label{sec:consecutive}
We begin with an interpolation estimate that will be used twice in the
proof of the main theorem. For $0\le\theta\le1$ and strictly positive
operators $X$ and $Y$, their weighted operator geometric mean is defined
by

\begin{equation}\label{eq:weighted-geometric-mean}
  X\#_\theta Y
  :=
  X^{1/2}
  \bigl(X^{-1/2}YX^{-1/2}\bigr)^\theta
  X^{1/2}.
\end{equation}
For arbitrary positive operators, the weighted geometric mean is defined
by the usual regularization procedure,
\begin{equation}\label{eq:weighted-geometric-mean-regularization}
  X\#_\theta Y
  :=
  \sslim_{\varepsilon\downarrow0}
  (X+\varepsilon I)\#_\theta(Y+\varepsilon I);
\end{equation}
see \cite[p.~533]{Yamazaki2019} and the Kubo--Ando theory (see
\cite{KuboAndo1980}). Here and throughout, $\sslim$ denotes convergence
in \textit{the strong operator topology}.

We shall use
the monotonicity of $\#_\theta$ in each variable and \textit{the transformer
inequality}
\begin{equation}\label{eq:transformer-geometric-mean}
  T^*(X\#_\theta Y)T
  \le
  (T^*XT)\#_\theta(T^*YT),
\end{equation}
valid for every bounded operator $T$ (see \cite[Theorem~3.5]{KuboAndo1980}).
If $X$ and $Y$ commute, then
\begin{equation}\label{eq:commuting-geometric-mean}
  X\#_\theta Y=X^{1-\theta}Y^\theta.
\end{equation}
We give two proofs of the following interpolation estimate. The first
uses the transformer inequality and monotonicity of the weighted
operator geometric mean, while the second relies directly on Hansen's
inequality.

\begin{lemma}\label{lem:power-interpolation}
Let $B,T\in\BH$, where $B\ge0$, and let $\alpha,\beta,r>0$.
If
\[
  T^*T\le B^\alpha
  \quad\text{and}\quad
  T^*B^rT\le B^\beta,
\]
then
\begin{equation}\label{eq:power-interpolation}
  T^*B^sT\le B^{(1-s/r)\alpha+(s/r)\beta},
  \qquad 0<s\le r.
\end{equation}
\end{lemma}

\begin{proof}[First proof of Lemma~\ref{lem:power-interpolation}]
Put $\theta=s/r\in(0,1]$. By the transformer inequality and
the monotonicity of the weighted geometric mean,
\[
  T^*B^sT
  =T^*(I\#_\theta B^r)T
  \le (T^*T)\#_\theta(T^*B^rT)
  \le B^\alpha\#_\theta B^\beta
  =B^{(1-\theta)\alpha+\theta\beta}.
\]
This proves \eqref{eq:power-interpolation}.
\end{proof}

\begin{proof}[Second proof of Lemma~\ref{lem:power-interpolation}]
For $\varepsilon>0$, put $B_\varepsilon=B+\varepsilon I$ and
$Q_\varepsilon=TB_\varepsilon^{-\alpha/2}$. Since
\[
  Q_\varepsilon^*Q_\varepsilon
  \le B_\varepsilon^{-\alpha}B^\alpha\le I,
\]
the operator $Q_\varepsilon$ is a contraction. Moreover,
\[
  Q_\varepsilon^*B^rQ_\varepsilon
  \le B_\varepsilon^{-\alpha}B^\beta
  \le B_\varepsilon^{\beta-\alpha}.
\]
Thus Hansen's inequality and the L\"owner--Heinz inequality give
\[
  Q_\varepsilon^*B^sQ_\varepsilon
  \le \bigl(Q_\varepsilon^*B^rQ_\varepsilon\bigr)^{s/r}
  \le B_\varepsilon^{(\beta-\alpha)s/r},
  \qquad 0<s\le r.
\]
Conjugating by $B_\varepsilon^{\alpha/2}$, we obtain
\[
  T^*B^sT
  \le B_\varepsilon^{(1-s/r)\alpha+(s/r)\beta}.
\]
Since the exponent on the right is positive, passage to the
limit in norm as $\varepsilon\downarrow0$ yields
\eqref{eq:power-interpolation}.
\end{proof}

We next establish the case of consecutive exponents, to which the proof
of Theorem~\ref{thm:all-pairs} will be reduced.

\begin{proposition}\label{thm:consecutive}
Let $m\ge2$ be an integer  and let $A\in\BH$. If $A^{*m}A^m=(A^*A)^m$ holds and
$A^{*(m+1)}A^{m+1}\le(A^*A)^{m+1}$, then $A$ is quasinormal.
\end{proposition}

\begin{proof}

Let $A=U|A|$ be the polar decomposition of $A$, and define
$\Phi\colon \boldsymbol B(\mathcal H)\to \boldsymbol B(\mathcal H)$ by
$\Phi(T)=U^*TU$. By hypothesis,
$A^*(A^*A)^mA\le (A^*A)^{m+1}$, and hence
\[
|A|\Phi((A^*A)^m)|A|
   \le (A^*A)^{m+1}.
\]
Since $\Phi((A^*A)^m)$ vanishes on $\jd{A^*A}$ and its range is contained in $\overline{\ran(A^*A)}$, conjugating the above inequality by $(A^*A+\varepsilon I)^{-1/2}$  
and passing to the limit in the strong operator topology
as $\varepsilon$ tends to zero, we obtain
\begin{align*}
    \Phi((A^*A)^m)\le (A^*A)^m.
\end{align*}
Therefore, applying Hansen's inequality to the operator monotone function
$f(t)=t^{s/m}$, and then using the operator monotonicity of this function, we obtain
\begin{equation}\label{eq:consecutive-comparison}
  \Phi((A^*A)^s)
  \le \Phi((A^*A)^m)^{s/m}
  \le (A^*A)^s,
  \qquad 0<s\le m.
\end{equation}
Multiplying on both sides by $|A|$ yields
$A^*(A^*A)^sA\le (A^*A)^{s+1}$. Since
$A^{*(j+1)}A^{j+1}=A^*(A^{*j}A^j)A$, 
the estimate above yields, by induction,
\[
A^{*(j+1)}A^{j+1}
\le A^*(A^*A)^jA
\le (A^*A)^{j+1},
\qquad 1\le j<m.
\]
Hence
\[
A^{*j}A^j\le (A^*A)^j,
\qquad 1\le j\le m.
\]
On the other hand, $A^{*m}A^m=(A^*A)^m,
$ gives
\[
|A|\Phi(A^{*(m-1)}A^{m-1})|A|
   =(A^*A)^m.
\]
The same regularization argument yields
$\Phi(A^{*(m-1)}A^{m-1})=(A^*A)^{m-1}$. Consequently,
\begin{equation}\label{eq:consecutive-fixed-moments}
\begin{aligned}
  (A^*A)^{m-1}
  &=\Phi(A^{*(m-1)}A^{m-1}) \\
  &\le \Phi((A^*A)^{m-1})
   \le \Phi((A^*A)^m)^{(m-1)/m}
   \le (A^*A)^{m-1}.
\end{aligned}
\end{equation}
Thus equality holds throughout. In particular,
$\Phi((A^*A)^s)=(A^*A)^s$ for $s=m-1,m$, and hence
$A^*(A^*A)^sA=(A^*A)^{s+1}$ for these two distinct positive exponents.
The conclusion follows from \Cref{bblem}, applied with
$T=A$ and $B=A^*A$.
\end{proof}

\begin{proof}[Proof of Theorem~\ref{thm:all-pairs}]
The forward implication follows from Embry's characterization.
For the converse, by homogeneity we may and shall assume that $\|A\|\le1$.
Put $k=n-m$ and, for $1\le r\le n$, define
\begin{equation}\label{eq:optimal-exponents}
  p_r=\max\bigl\{
    p\in[1,r]: A^{*r}A^r\le(A^*A)^p
  \bigr\}.
\end{equation}
The maxima in \eqref{eq:optimal-exponents} are well defined. Indeed, since $A$ is a contraction, one has $A^{*r}A^r\le A^*A$, so the corresponding sets are nonempty.
 Moreover, the map $p\longmapsto (A^*A)^p$
is norm continuous on $[1,r]$. Hence each set in
\eqref{eq:optimal-exponents} is compact, and therefore its maximum is
attained. Consequently,
\[
  p_1=1,
  \qquad
  1\le p_r\le r,
  \qquad 1\le r\le n.
\]

We first establish two inequalities for the optimal exponents:
\begin{align}
  p_j
  &\ge
  \frac{(m-1)p_{j-1}+p_{m+j-1}}{m},
  &&2\le j\le k,
  \label{eq:optimal-feedback}\\
  p_{k+\ell}
  &\ge p_k+p_\ell,
  &&1\le\ell\le m.
  \label{eq:optimal-additivity}
\end{align}

Fix $2\le j\le k$. From the first inequality in
\eqref{eq:main-hypotheses} and the definition of
$p_{m+j-1}$, we obtain
\[
  (A^{j-1})^*(A^*A)^mA^{j-1}
  \le A^{*(m+j-1)}A^{m+j-1}
  \le (A^*A)^{p_{m+j-1}}.
\]
Applying \Cref{lem:power-interpolation} with
$T=A^{j-1}$, $B=A^*A$, $\alpha=p_{j-1}$,
$\beta=p_{m+j-1}$, $r=m$ and $s=1$, yields
\[
  A^{*j}A^j
  \le
  (A^*A)^{((m-1)p_{j-1}+p_{m+j-1})/m}.
\]
Since
\[
  \frac{(m-1)p_{j-1}+p_{m+j-1}}{m}
  \le
  \frac{(m-1)(j-1)+(m+j-1)}{m}
  =j,
\]
the definition of $p_j$ gives
\eqref{eq:optimal-feedback}.

Next, let $1\le\ell\le m$. Since $k=n-m$, the two
inequalities in \eqref{eq:main-hypotheses} imply
\[
  (A^k)^*(A^*A)^mA^k
  \le A^{*n}A^n
  \le (A^*A)^n.
\]
As $0<p_\ell\le\ell\le m$, another application of
\Cref{lem:power-interpolation}, now with
$T=A^k$, $\alpha=p_k$, $\beta=n$, $r=m$ and $s=p_\ell$, gives
\begin{align*}
  A^{*(k+\ell)}A^{k+\ell}
  &=(A^k)^*(A^{*\ell}A^\ell)A^k\\
  &\le (A^k)^*(A^*A)^{p_\ell}A^k\\
  &\le (A^*A)^{\,p_k+(n-p_k)p_\ell/m}\\
  &\le (A^*A)^{p_k+p_\ell}.
\end{align*}
Here the last inequality follows from
$n-p_k\ge n-k=m$ and $0\le A^*A\le I$.
Since $p_k+p_\ell\le k+\ell$, this proves
\eqref{eq:optimal-additivity}.

Set
\[
  \theta=\min_{1\le r\le n}\frac{p_r}{r}.
\]
We claim that $\theta=1$. Suppose otherwise, and let $j$ be the
smallest index for which $p_j=\theta j$. Since $p_1=1$, we have
$j\ge2$.

If $j\le k$, then $p_{j-1}>\theta(j-1)$, whereas
$p_{m+j-1}\ge\theta(m+j-1)$. Hence
\eqref{eq:optimal-feedback} gives
\[
  p_j
  \ge
  \frac{(m-1)p_{j-1}+p_{m+j-1}}{m}
  >
  \frac{(m-1)\theta(j-1)+\theta(m+j-1)}{m}
  =\theta j=p_j,
\]
a contradiction.

If $j>k$, then both $k$ and $j-k$ are strictly smaller than $j$,
and $1\le j-k\le m$. By the minimality of $j$,
\[
  p_k>\theta k,
  \qquad
  p_{j-k}>\theta(j-k).
\]
Using \eqref{eq:optimal-additivity}, we obtain
\[
  p_j
  \ge p_k+p_{j-k}
  >
  \theta k+\theta(j-k)
  =\theta j=p_j,
\]
again a contradiction. Thus $\theta=1$.

It follows that $p_r=r$ for every $1\le r\le n$. In particular,
\[
  A^{*m}A^m\le(A^*A)^m
  \quad\text{and}\quad
  A^{*(m+1)}A^{m+1}\le(A^*A)^{m+1}.
\]
The first of these, together with the first inequality in
\eqref{eq:main-hypotheses}, gives
$A^{*m}A^m=(A^*A)^m$. The conclusion now follows from
\Cref{thm:consecutive}.

\end{proof}

\section{Proof of Theorem~\ref{thm:flatness}}\label{sec:flatness}

We apply Lemma~\ref{lem:power-interpolation} to give an affirmative
answer to the flatness problem posed in
\cite[Problem~4.2]{PietrzyckiStochel2023}. We first establish a
criterion involving an auxiliary positive operator.

\begin{proposition}\label{prop:flatness-criterion}
Let $T,B\in\BH$, with $B\ge0$, and let $n\ge2$ be an integer.
Suppose that $T^*T\le B$ and
\[
  T^{*k}T^k=B^k,
  \qquad k\in\{n,n+1\}.
\]
Then $B=T^*T$ and $T$ is quasinormal.
\end{proposition}

\begin{proof}
The assumed identities imply that
\[
  T^*B^nT=T^{*(n+1)}T^{n+1}=B^{n+1}.
\]
Applying Lemma~\ref{lem:power-interpolation} with
$(\alpha,\beta,r)=(1,n+1,n)$, we obtain
\begin{equation}\label{eq:flatness-interpolation}
  T^*B^sT\le B^{s+1},
  \qquad 0<s\le n.
\end{equation}
Since $T^*T\le B$, induction using
\eqref{eq:flatness-interpolation} gives
\[
  T^{*j}T^j\le B^j,
  \qquad 1\le j\le n.
\]
Indeed, if $T^{*j}T^j\le B^j$ for some $1\le j<n$, then
\[
  T^{*(j+1)}T^{j+1}
  =T^*(T^{*j}T^j)T
  \le T^*B^jT
  \le B^{j+1}.
\]
Consequently,
\[
  B^n
  =T^*(T^{*(n-1)}T^{n-1})T
  \le T^*B^{n-1}T
  \le B^n,
\]
and hence
\[
  T^*B^sT=B^{s+1},
  \qquad s\in\{n-1,n\}.
\]
As $n-1$ and $n$ are distinct positive numbers,
Theorem~\ref{bblem} yields $B=T^*T$ and the quasinormality of $T$.
\end{proof}

\begin{proof}[Proof of Theorem~\ref{thm:flatness}]
Put $D=(A^{*n}A^n)^{1/n}$. By
\cite[Theorem~2.2]{PietrzyckiStochel2023}, the sequence
$\{(A^{*j}A^j)^{1/j}\}_{j=1}^{\infty}$ is monotonically increasing,
so $A^*A\le D$. The assumed equality gives
$A^{*k}A^k=D^k$ for $k\in\{n,n+1\}$.
Proposition~\ref{prop:flatness-criterion} now implies that $A$ is
quasinormal. 
\end{proof}

\end{document}